\documentclass[12pt]{amsart}
\usepackage{amssymb}

\usepackage{tikz}
\newtheorem{Teo}{Theorem}[section]
\newtheorem*{theorem*}{Theorem}
\newtheorem{Lem}[Teo]{Lemma}
\newtheorem{Obs}[Teo]{Remark}
\newtheorem{Def}[Teo]{Definition}

\newtheorem{Cor}[Teo]{Corollary}

\newcommand{\VR}{\mathcal{O}}
\newtheorem{Con}[Teo]{Conjecture}
\newtheorem{lem-def}[Teo]{Lemma-Definition}

\DeclareRobustCommand\longtwoheadrightarrow
{\relbar\joinrel\twoheadrightarrow}

\newcommand{\hooklongrightarrow}{\lhook\joinrel\longrightarrow}
\renewenvironment{proof}{{\bfseries Proof.}}{\qed}
\newcommand{\Ha}{\mathbb H}
\newcommand{\R}{\mathbb R}

\newcommand{\N}{\mathbb N}

\newcommand{\Q}{\mathbb Q}

\newcommand{\F}{\mathbb F}

\def\op{\operatorname}

\def\al{\alpha}

\def\ars#1{\renewcommand\arraystretch{#1}}
\def\as{\op{AS}}
\def\aut{\op{Aut}}

\def\be{\beta}

\def\comb#1#2{\ars{0.9}\left(\!\!\begin{array}{c}
#1\\#2
\end{array}\!\!\right)\ars{1}}

\def\dep{\op{\mbox{\rm\small depth}}}
\def\deta{\deg_\eta}

\def\diso{\lower.4ex\hbox{$\downarrow$}\raise.4ex\hbox{\mbox{\scriptsize
$\wr$}}}

\def\dta{\delta}

\def\e{\medskip}

\def\ep#1{\exp(\Pi i#1)}
\def\ep{\epsilon}

\def\g{\Gamma}
\def\ga{\gamma}

\def\gal{\op{Gal}}

\def\gen#1{\big\langle\, {#1} \,\big\rangle}

\def\gi{\g_{\infty}}

\def\imp{\ \Longrightarrow\ }

\def\ism{\lower.3ex\hbox{\ars{.08}$\begin{array}{c}\,\to\\\mbox{\tiny $\sim\,$}\end{array}$}}
\def\iso{\ \lower.3ex\hbox{\ars{.08}$\begin{array}{c}\lra\\\mbox{\tiny $\sim\,$}\end{array}$}\ }

\def\ka{\kappa}

\def\kb{\overline{K}}
\def\kbx{\overline{K}[x]}

\def\ks{K^{\op{sep}}}

\def\kx{K[x]}

\def\la{\lambda}

\def\lg{l\raise.6ex\hbox to.2em{\hss.\hss}l}

\def\lm{\op{Lim}}
\def\lra{\,\longrightarrow\,}

\def\om{\omega}

\def\orb{\hbox to  .3em{$\backslash$}\backslash}
\def\ord{\op{ord}}

\def\Ri{\R_\infty}

\def\sg{\sigma}

\def\sii{\ \Longleftrightarrow\ }

\def\sub{\subseteq}
\def\supp{\op{supp}}

\def\t{\theta}

\def\tr{\op{trn}}
\def\Tr{\op{Tr}}

\def\vt{v_\theta}

\def\z{\op{Z}}

\newcounter{cs}
\stepcounter{cs}
\newcommand{\casos}{\begin{itemize}}
\newcommand{\fcasos}{\end{itemize}\setcounter{cs}{1}}

\newfont{\tit}{cmr12 scaled \magstep3}

\title{Depth of Artin-Schreier defect towers}

\makeatletter
\@namedef{subjclassname@2010}{%
  \textup{2020} Mathematics Subject Classification}
\author[Nart]{Enric Nart}
\address{Departament de Matem\`{a}tiques,         Universitat Aut\`{o}noma de Barcelona,         Edifici C, E-08193 Bellaterra, Barcelona, Catalonia}
\email{enric.nart@uab.cat}

\author[Novacoski]{Josnei Novacoski}
\address{Departamento de Matem\'{a}tica,         Universidade Federal de S\~ao Carlos, Rod. Washington Luís, 235, 13565--905, S\~ao Carlos -SP, Brazil}
\email{josnei@ufscar.br}

\thanks{Partially supported by grant PID2020-116542GB-I00  funded by the Spanish MCIN/AEI. During the realization of this project the second author was supported by a grant from Funda\c{c}\~ao de Amparo \`a Pesquisa do Estado de S\~ao Paulo (process number 2024/08989-6) and a grant from Conselho Nacional de Desenvolvimento Cient\'ifico e Tecnol\'ogico (process number 303215/2022-4).}


\begin{document}

\begin{abstract}
The depth of a simple algebraic extension $(L/K,v)$ of valued fields is the minimal length of the Mac Lane-Vaquié chains of the valuations on $\kx$ determined by the choice of different generators of the extension. In \cite{NN25}, Nart and Novacoski characterized the defectless unibranched extensions of depth one. In this paper, we analyze this problem for towers of Artin-Schreier defect extensions. Under certain conditions on $(K,v)$, we prove that the towers obtained as the compositum of linearly disjoint defect Artin-Schreier extensions of $K$ have depth one. We conjecture that these are the only depth one Artin-Schreier defect towers and we present some examples supporting this conjecture. 
\end{abstract}

\maketitle



\section*{Introduction}
Let $K$ be a field, $\kb$ an algebraic closure of $K$, and $v$ a valuation on $\kb$ such that  $(K,v)$ is Henselian. For any subfield $K\sub L\sub \kb$, we denote the \textit{value group} and \textit{residue field} of $(L,v)$ by $vL$ and $Lv$, respectively. 

For $\t\in\overline{K}$ we consider the valuation $v_\theta$ on $K[x]$ defined as
\[
\vt(f)=v(f(\theta)).
\]
The length  of the \emph{Mac Lane-Vaqui\'e (MLV) chains} of $\vt$  is said to be the \textbf{depth} of  $\vt$,  denoted by $\dep(\t)$ (see Section \ref{secOS} for more details). This concept is not intrinsically associated to $(L/K,v)$. Different  generators of the same extension may have different depths. This leads to defining the depth of an extension of valued fields as
\[
\dep(L/K,v):=\min\{\dep(\t)\mid L=K(\t)\}.
\]

In \cite{OS}, a tight link was established between MLV chains of $\vt$ and \emph{Okutsu sequences} of $\t$. These ideas were used in \cite{NN25} to show that Okutsu sequences  facilitate the computation of the depth of generators of extensions of valued fields. 
In particular, this tool was used to compute the depth  of several defectless extensions and to characterize the defectless extensions with depth equal to one.

The aim of this paper is to use Okutsu sequences to study $\dep(L/K,v)$ for extensions with defect. The defect is known to be one of the main obstacles for many problems in algebraic geometry, model theory and number theory, among others. Some of these problems are \emph{resolution of singularities in positive characteristic} (see \cite{Nov4}), the \emph{decidability of $\F_p((t))$} (see \cite{KR}) and \emph{polynomial factorization over Henselian fields} (see \cite{OM}).

One particular motivation to study the depth of defect extensions comes from the study of the module $\Omega$ of K\"ahler differentials for the extension $\VR_L/\VR_K$ of the corresponding valuation rings. These objects are crucial for the understanding of \emph{perfectoid fields} and \emph{deeply ramified fields} (see \cite{GabRam}). By the results of \cite{NS2026}, if $\dep(L/K,v)=1$, then   $\Omega$ has a very explicit structure. A similar result appears in \cite[Theorem 1.1]{CKR}. This result gives explicit formulas for $\Omega$, obtained from a suitable set of generators of $\VR_L/\VR_K$. It is easy to show that if $(L/K,v)$ is a defect extension of depth one, then such set of generators exists.

Hence, it is important to find criteria for the extension $(L/K,v)$ to have depth one. If the extension has defect, then the depth-one property amounts to the existence of some  $\t\in L$ such that $L=K(\t)$ and $\vt$ admits an MLV chain
\[
v\,\lra\,\mu_0\,\lra\,\vt, 
\]  
where $\mu_0\to\vt$ is  a \emph{limit augmentation}. In the language of \cite{CKR}, this is the same as saying that $(L/K,v)$ is \emph{pure} in $\theta$.

The outline of the  paper is as follows. In Section \ref{secOS} we recall Okutsu sequences of algebraic elements $\t\in\kb$  and its connection with the depth of $\vt$. In Section \ref{secAS} we focus on the behaviour of the depth on towers of Artin-Schreier (abbreviated as AS) defect extensions. Namely, we consider a tower
\[
K= K_0\subseteq\ldots \subseteq K_n=M
\] 
where $K_{i}/K_{i-1}$ is an AS extension for every $i$, $1\leq i\leq n$. This is again motivated by \cite{CKR}. The main result of that section (Theorem \ref{generaliza}) provides a condition for the compositum of finitely many AS defect extensions to have depth one. In the rank one and Henselian case (and under a mild extra condition), this shows that the compositum of finitely many \emph{independent} AS defect extensions has depth one.  We also present a criterion (Lemma \ref{criterion}) for a tower of two AS extensions (not necessarily with defect) to be Galois.

In Section \ref{subsethanh}, we focus on the case where $K$ is a subfield of the Hahn field $\mathbb H$. This situation provides many tools that simplify the study of the behaviour of the depth on towers of AS defect extensions. We conjecture (Conjecture \ref{mainCon}) that for a subfield $K$ of $\mathbb H$ which is Henselian, perfect and \emph{has no finite limits}, the only towers of AS defect extensions of depth one are the compositum of finitely many AS defect extensions.

In Section \ref{examples} we present some examples supporting Conjecture \ref{mainCon}. The example built in Section \ref{subsecEx2} consists of a tower $M/K$ of two AS defect extensions such that $M/K$ is Galois, ${\rm Gal}(M/K)$ is cyclic and $\dep(M/K)=2$. In Section \ref{subsecEx3} we present a similar example for which $M/K$ is non-Galois. Finally, in Section \ref{Exemp3} we present an example of a tower $N/K$ of three AS defect extensions, which is Galois, its Galois group is not abelian and $\dep(N/K)=2$.

\section{Okutsu sequences and depth}\label{secOS}
In this section, we do not make any assumption on the valued field $(K,v)$. Let us still denote by $v$ some fixed extension of $v$ to an algebraic closure $\kb$ of $K$. Denote $\g:=v\kb$.
From now on, we denote $\g\cup\{\infty\}$ simply by $\gi$.

For some $\t\in\kb$, let  $g\in\kx$ be its minimal  polynomial over $K$ and consider the extension $L=K(\t)$ of $K$. S. Mac Lane realized that the properties of the valued field $(L,v)$  could be described in terms of the following valuation on $\kx$:
\[
\vt\colon \kx\lra \g\cup\{\infty\},\qquad f\longmapsto \vt(f)=v(f(\t)).
\]  
Consider  the isomorphism $\kx/(g)\simeq L $ induced by $x\mapsto \t$. Since $\vt^{-1}(\infty)=g\kx$, the valuations  $\vt$ and $v_{\mid L}$ are determined one by each other through
\[
\vt\colon \kx \longtwoheadrightarrow \kx/(g)\stackrel{\sim}\lra L\stackrel{v}\lra \g\cup\{\infty\}.
\]

A celebrated theorem of Mac Lane-Vaqui\'e states that $\vt$ can be constructed from $v$ by means of an \emph{MLV chain}; that is, a finite sequence of \textit{augmentations} of valuations on $\kx$ whose restriction to $K$ is $v$:
\[
v\ \lra\ 	\mu_0\ \lra\  \mu_1\ \lra\ \cdots
	\ \lra\ \mu_{r}=\vt,
\]
satisfying certain natural properties \cite{Vaq, MLV}. The initial augmentation $v\to\mu_0$ is symbolic; it only indicates that $\mu_0$ is a \emph{monomial} valuation. Each of the real augmentations $\mu_n\to\mu_{n+1}$ can be either \emph{ordinary} or \emph{limit}. 
The valuation $\vt$ admits different MLV chains, but all of them have the same length $r$ and the same sequence of characters ordinary/limit of the successive augmentations \cite[Section 4]{MLV}.


Let $n\ge1$ be the degree of $\t$ over $K$. 
For every integer $1\le m\le n$, we define the \textbf{set of distances} of $\t$ to elements in $\kb$ of degree $m$ over $K$ as:
\[
D_m=D_m(\t,K):=\left\{v(\t-b)\mid b\in\kb,\ \deg_Kb=m\right\}\sub\gi.
\] 
Note that $\max(D_n)=\infty$.

The set $D_1$ has been extensively studied by Blaszczok and Kuhlmann for its connections with defect and immediate extensions \cite{B,Kuhl}.

We are interested in $\max\left(D_m\right)$, the maximal distance of $\t$ to elements of a fixed degree over $K$. Since this maximal distance may not exist, we follow \cite{Kuhl} and we replace this concept with an analogous one in the context of cuts in $\g$.

A \textbf{cut} in $\g$ is a pair $\dta=(\dta^L,\dta^R)$ of subsets of $\g$ such that $$\dta^L< \dta^R\quad\mbox{ and }\quad \dta^L\cup \dta^R=\g.$$ 
The inequality $\dta^L< \dta^R$ means that $\al<\be$ for all elements $\al\in \dta^L$, $\be\in \dta^R$. Note that 
 $\dta^R=\g\setminus\dta^L$. Let us denote by $\op{Cuts}(\g)$ the set of all cuts in $\g$. 

For all $D\sub \g$ we denote by  $D^+$, $D^-$ the cuts determined by
\[
\dta^L=\{\ga\in \g\mid \exists \al\in D: \ga\leq \al\},\qquad \dta^L=\{\ga\in \g\mid \ga<D\},
\]
respectively.
If $D=\{\ga\}$, then we will write $\ga^+=(\g_{\le \ga},\g_{>\ga})$ instead of $\{\ga\}^+$ and $\ga^-=(\g_{<\ga}, \g_{\ge \ga})$ instead of $\{\ga\}^-$. These cuts are said to be 
\textbf{principal}.

The set $\op{Cuts}(\g)$  is totally ordered with respect to the following ordering:
\[
(\dta^L,\dta^R)\le (\ep^L,\ep^R)\ \sii\ \dta^L\sub \ep^L.
\]
The \textbf{improper} cuts 
$-\infty:=(\emptyset,\g)$, $\infty^-:=(\g,\emptyset)$ 
are the absolute minimal and maximal elements in $\op{Cuts}(\g)$, respectively.

\begin{Def}\label{distance}
For $m<n$, we define $d_m(\t)$ to be the cut $D_m^+\in \op{Cuts}(\g)$.  

We agree that $d_n(\t)$ is the improper cut $\infty^-$.
\end{Def}

For instance, if $D_m$ has a maximal element $\ga\in\g$, then  $d_m(\t)=\ga^+$.

As mentioned  above, the valuation $\vt$ on $\kx$ contains relevant information about the valued field $(L,v)$. This information can be captured as well by certain sequences of sets of algebraic elements. 

We say that a  subset $A\sub\kb$ has a \textbf{common degree} if all its elements have the same degree over $K$. In this case, we shall denote this common degree by $\deg_K A$. 

\begin{Def}\label{defOkS}
	An \textbf{Okutsu sequence} of $\t$ is a finite sequence 
\[
\left[A_0,A_1,\dots,A_{r-1},A_r=\{\t\}\right], 
\]
of common degree subsets of $\kb$ whose degrees grow strictly:
\begin{equation}\label{degOS}	
		1=m_0<m_1<\cdots<m_r=n,\qquad m_\ell=\deg_K A_\ell, \ \ 0\le\ell\le r,
\end{equation}
and	satisfy the following properties for all $\,0\le\ell< r$: \e
	
	(OS0) \ For all $b\in\kb$ such that $\deg_Kb<m_{\ell+1}$, we have $v(\t-b)\le  v(\t-a)$ for some $a\in A_\ell$.\e
	
	
	(OS1) \ $\#A_\ell=1$ whenever $\max\left(D_{m_{\ell}}\right)$ exists.\e
	
	(OS2) \ If $\max\left(D_{m_{\ell}}\right)$ does not exist, then we assume that $A_\ell$ is well-ordered with respect to the following ordering: \ $a<a'\ \sii \ v(\t-a)<v(\t-a')$.\e
	
	(OS3) \ For all $a\in A_{\ell}$, $b\in A_{\ell+1}$, we have $v(\t-a)<v(\t-b)$.
\end{Def}

Let us discuss the existence and construction of Okutsu sequences.
Consider the sequence of minimal degrees of the distances of $\t$:
\begin{equation}\label{degDist}	
1=d_0<d_1<\cdots<d_s=n,
\end{equation}
defined recursively as follows:

$\bullet$ \ $d_0=1$,

$\bullet$ \ for every $\ell\in\N$, $d_\ell$ is the least integer $m>d_{\ell-1}$ such that there exists some $\ep\in D_m$ satisfying $\ep>D_{d_{\ell-1}}$.\e

Now, for each $0\le \ell<  s$ choose subsets $A_\ell\sub\kb$ of common degree $\deg_K A_\ell=d_\ell$ such that
\[
 \{v(\t-a)\mid a\in A_\ell\}\sub D_{d_\ell}
\]
is a well-ordered cofinal subset of $D_{d_\ell}$. Also, if for some $\ell$ there exists $\ga=\max\left(D_{d_\ell}\right)$, then we take $A_\ell=\{a\}$, for some $a\in\kb$ such that $\deg_K a=d_\ell$ and $v(\t-a)=\ga$. 

Finally, for $\ell>0$, we consider in $A_\ell$ only elements $a$ such that $v(\t-a)>D_{d_\ell-1}$. 

Clearly, $\left[A_0,A_1,\dots,A_{r-1},A_r=\{\t\}\right]$ is an Okutsu sequence of $\t$ and, conversely, all Okutsu sequences arise in this way. 

In particular, the sequence (\ref{degOS}) of degrees over $K$ of the sets $A_\ell$ is equal to the canonical sequence (\ref{degDist}) of minimal degrees of the distances of $\t$. That is, $r=s$ and $d_\ell=m_\ell$ for all $0\le \ell<r$.

The link between $\dep(\t)$  and Okutsu sequences of $\t$ is established in
the following theorem, which  was proved in \cite{OS} under the assumption that $(K,v)$ is Henselian. For arbitrary valued fields, it was shown in  \cite{NN25} that this  result follows easily from \cite[Theorem 7.2]{NNP}. 

\begin{Teo}\label{OSdepth}
The length $r$ of any Okutsu sequence of $\t$ is equal to $\dep(\t)$. Moreover, for every MLV chain of $\vt$:
\[
v\ \to\ \mu_0\ \to\ \mu_1\ \to\ \cdots \ \to\ \mu_{r-1}\ \to\ \mu_r=\vt, 
\]
and every $0\le \ell<r$,
the augmentation $\mu_\ell\, \to\, \mu_{\ell+1}$ is ordinary if  and only if $D_{m_\ell}$ contains a maximal element.
\end{Teo}

\section{Artin-Schreier defect extensions}\label{secAS}
Let $K$ be any field of positive characteristic. 
\subsection{Background on AS extensions}
In this section, we study Artin-Schreier extensions of $K$. We shall often abbreviate the qualifier ``Artin-Schreier" by AS.
Consider the $\F_p$-linear operator
\[
\as\colon \kb\; \longtwoheadrightarrow\; \kb,\qquad \as(\al)=\al^p-\al.
\]
Clearly, $\ker(\as)=\F_p$. In particular, if $\z(f)\sub \kb$ denotes  the set of zeros of a polynomial $f\in\kbx$, we have \[\z(x^p-x-\as(\al))=\al+\F_p,\quad\mbox{for all }\  \al\in\kb.
\]

\begin{Def}
	Suppose that $\al\in\kb$ satisfies $\as(\al)\in K\setminus \as(K)$. Then, we say that $L=K(\al)$ is an \textbf{AS extension} of $K$ and $\al$ is	 an \textbf{AS generator} of $L/K$. 
\end{Def}

The following observation can be found in \cite{Kuhl}.
\begin{Lem}\label{AllGen}
	Let $L/K$ be an AS extension with AS generator $\al\in L$. Then, $\F_p^*\,\al+K$ is the set of all AS generators of $L/K$.  
\end{Lem}

An AS extension $L/K$ is Galois with cyclic Galois group $C_p$, generated by the automorphism $\sg$ determined by  $\sg(\al)=\al+1$, on an arbitrary AS generator  $\al$ of $L/K$.

\subsection{Compositum of AS extensions}

Let $\al_1,\dots,\al_r\in\kb$ be AS elements such that the extensions $K(\al_1),\dots,K(\al_r)$  are  linearly disjoint over $K$.

The compositum $L=K(\al_1,\dots,\al_r)$ is a Galois extension with Galois group generated by the automorphisms $\sg_1,\dots,\sg_r$ determined by
\[
\sg_i(\al_j)=\begin{cases}
\al_i+1,&\mbox{ if }i=j,\\
\al_j,&\mbox{ if }i\ne j, 
\end{cases}\qquad \mbox{ for all }\ i=1,\dots,r.
\]
Since these automorphisms commute, we have $G:=\op{Gal}(L/K)\simeq C_p^r$.

The subgroups $H\sub G$ such that $(G\colon H)=p$ can be identified to 1-codimensional $\F_p$-vector subspaces of $C_p^r$.  Hence, there are $(p^r-1)/(p-1)$ subgroups of index $p$. It is easy to check that the corresponding subextensions of $L/K$ of degree $p$ over $K$ are the AS extensions generated by
\begin{equation}\label{gens}
\ell_1\al_1+\cdots+\ell_r\al_r,\qquad  (\ell_1,\dots,\ell_r)\in\F_p^r\setminus \{(0,\dots,0)\}.
\end{equation}
Two vectors  $(\ell_1,\dots,\ell_r),(\ell'_1,\dots,\ell'_r)\in\F_p^r$ determine the same AS extension if and only if $(\ell'_1,\dots,\ell'_r)=
\la(\ell_1,\dots,\ell_r)$
for some $\la\in\F_p^*$.

\begin{Lem}\label{compos}
Let $a_1,\dots, a_n\in K$ and \,$\al_1,\dots,\al_n\in\kb$, such that
\[
\as(\al_i)=a_i,\quad \mbox{ for all }\ i=1,\dots,n.
\]
Then, the following conditions are equivalent.
\begin{enumerate}
	\item [(a)] \ $K(\al_1),\dots, K(\al_n)$ are linearly disjoint over $K$. 
	\item [(b)] \ For all $i=1,\dots,n$, we have: \ $\al_i\not\in
	\gen{\al_1,\dots,\al_{i-1},\al_{i+1},\dots,\al_n}_{\F_p}+K$.
	\item [(c)] \ The images of $a_1,\dots,a_n$ in $K/\as(K)$ are $\F_p$-linearly independent.
\end{enumerate}
\end{Lem}

\begin{proof}
Let us first show that (a) and (b) are equivalent. The implication (a)$\Rightarrow$(b) being obvious, let us show that (b) implies (a).

Suppose that, for some $r<n$, the extensions $K(\al_1),\dots, K(\al_r)$ are linearly disjoint over $K$ and $K(\al_n)\sub L:=K(\al_1,\dots,\al_r)$. Then, $\al_n$ generates the same AS extension than one of the AS generators described in (\ref{gens}). By Lemma \ref{AllGen}, $\al_n$ belongs to $\gen{\al_1,\dots,\al_r}_{\F_p}+K$, contradicting (b).

Now, let us prove that (b) and (c) are equivalent. Suppose that, for instance,
\[
\al_n=\ell_1\al_1+\cdots+\ell_{n-1}\al_{n-1}+a
\]
for some $a\in K$ and $\ell_1,\dots,\ell_{n-1}\in\F_p$. By applying the $\F_p$-linear operator $\as$, we get
\begin{equation}\label{notLI}
a_n=\ell_1a_1+\cdots+\ell_{n-1}a_{n-1}+\as(a),
\end{equation}
contradicting (c). This shows that  (c)$\Rightarrow$(b).

Conversely, an equality as in (\ref{notLI}) implies that
\[
\as(\al_n)=\as\left(\ell_1\al_1+\cdots+\ell_{n-1}\al_{n-1}+a\right).
\]
This implies that $\al_n\in\gen{\al_1,\dots,\al_{n-1}}_{\F_p}+K$. Hence, (b) implies (c).
\end{proof}

\subsection{Main theorem}
\begin{Teo}\label{generaliza}
Suppose that $(K,v)$ is a Henselian field. Let $\alpha_1,\ldots,\alpha_n\in \overline{K}$ be generators of AS extensions such that $K(\alpha_1),\ldots,K(\alpha_n)$ are linearly disjoint over $K$. Suppose that \begin{equation}\label{equline}
\exists c_1,\ldots,c_n\in K: (a_1,\ldots,a_n)\in \F_p^n\setminus \{(0,\ldots,0)\}\Longrightarrow  v(c_1a_1+\ldots+c_na_n)=0.
\end{equation}
If $d_1(\alpha_i)=0^-$ for every $i=1,\ldots, n$, then $L=K(\alpha_1,\ldots,\alpha_n)$ has depth one over $K$.
\end{Teo}

\begin{proof}
Consider
\[
\theta:=c_1\alpha_1+\cdots +c_n\al_n\in L.
\]	
The Galois conjugates of $\t$ over $K$ are
\[
\t+c_1\F_p+\cdots+c_n\F_p.
\]
Since this set has cardinality $p^n$, we see that $\t$ is a generator of $L/K$. Also, note that Krasner's constant of $\t$ is zero:
\[
\om(\t):=\max\{v(\t-\sg(\t))\mid \sg\in G, \sg\ne1\}=0.
\]

Since $d_1(\al_1)=\cdots=d_1(\al_n)=0^-$, we easily deduce that $d_1(\theta)\geq 0^-$. Now, for every $\xi\in\kb$, the inequality
\[
v(\t-\xi)>d_1(\theta)\geq 0^-
\]
implies that $K(\t)\sub K(\xi)$ by Krasner's Lemma. Hence, $\deg_K \xi\ge p^n$ and this implies that $\t$ admits an Okutsu sequence of length one. By Theorem \ref{OSdepth}, $\t$ has depth one and this concludes our proof.
\end{proof}
\begin{Obs}
Observe that condition \eqref{equline} is easily satisfied. Such $c_i$'s exist if we assume for instance that $K$ contains an infinite algebraic extension of $\F_p$.
\end{Obs}
The above result can be rephrased in terms of the classification of AS defect extensions presented in \cite{Kuhl}.
\begin{Def}
Suppose that $(K,v)$ is a rank one valued field. An AS defect extension $(K(\alpha)/K,v)$ is  said to be \textbf{independent} if $d_1(\alpha)=0^-$. Otherwise, it is called \textbf{dependent}.
\end{Def}

\begin{Cor}
Suppose that $(K,v)$ is a rank one Henselian field and that condition \eqref{equline} is satisfied. Then, the compositum of finitely many AS independent defect extensions has depth one.
\end{Cor}
\subsection{AS towers of degree $p^2$}
A finite field extension $M/K$ is said to be an \textbf{AS tower} if it can be constructed as a chain of extensions
\[
K=L_0\,\sub\,L_1\,\sub\,\cdots\,\sub\,L_n=M
\] 
in which every step $L_{i-1}\sub L_i$ is an AS extension.
A compositum of AS extensions of $K$ can be seen as an AS tower in an obvious way.

Let $\ks\sub \kb$ be the separable closure of $K$. Consider the absolute Galois group
\[
\mathbb{G}:=\gal(\ks/K).
\]
Let us fix an AS extension $L/K$ with AS generator $\al$. We shall denote \[a:=\as(\al)\in K\setminus \as(K).\]

Also, we fix an automorphism  $\sg\in \mathbb{G}$ such that $\sg(\al)=\al +1$.
As mentioned in the previous section, $L/K$ is a Galois extension  with  cyclic Galois group  generated by the restriction of $\sg$ to $L$.  Every $\be\in L$ has a trace over $K$ given by 
\[
\Tr_{L/K}(\be):=\be+\sg(\be)+\cdots +\sg^{p-1}(\be).
\]
Note that $\Tr_{L/K}(K)=\{0\}$.

Consider an AS extension of $L$ with AS generator $\eta\in\kb$:
\[
K\sub L=K(\al)\sub M=L(\eta).
\]
Denote $b:=\as(\eta) \in L\setminus \as(L)$.
The following criterion for $M/K$ to be a Galois extension can be deduced from some topics discussed in \cite{Lorenz}. For the ease of the reader we provide a short proof.

\begin{Lem}\label{criterion}
	The extension $M/K$ is Galois 	if and only if $\sg(b)-b\in \as(L)$.
\end{Lem}

\begin{proof}
	Since $M/K$ is a separable extension, we  bother only about normality. 
	Suppose that $M/K$ is a normal extension. Then, $\sg(\eta)\in M$. Since \[\as(\sg(\eta))=\sg(b)\in L\setminus \as(L),\] the element $\sg(\eta)$ is an AS generator of $M/L$.  By Lemma \ref{AllGen}, 
	\[
	\sg(b)\in\ell b+\as(L),\quad\mbox{for some }\ \ell\in \F_p^*.
	\]	
	A successive application of the automorphism $\sg$ leads to
	\[
	b=\sg^p(b)\in\ell^pb+\as(L)=\ell b+\as(L).
	\]
	This implies that $\ell=1$, because otherwise we get a contradiction:
	\[
	(1-\ell)b\in\as(L),\ \ell\ne 1\imp b\in\as(L).
	\]
	Therefore, we have seen that $\sg(b)-b\in \as(L)$.
	
	Conversely, suppose now that  $\sg(b)-b=\as(\be)$ for some $\be\in L$. 	
	Take any $\tau\in \mathbb{G}$. We shall have $\tau_{\mid L}=\sg^m_{\mid L}$ for some  $m\in\N$. Hence,
	\[
	\tau(b)-b=\sg^m(b)-b=\as(\be_m),\quad\mbox{for some }\ \be_m\in \as(L).
	\] 
Since $\as(\tau(\eta)-\eta-\be_m)=\tau(b)-b-\as(\be_m)=0$, we have $\tau(\eta)-\eta-\be_m\in \F_p$, so that $\tau(\eta)\in M$. 
	This shows that $M/K$ is a normal extension.
\end{proof}\e

Let us now study the structure of $G:=\gal(M/K)$ in the Galois case. Thus,  from now on we suppose that 
\[\sg(b)-b=\as(\be),\quad\mbox{for some } \ \be\in L.\]

Since $\Tr_{L/K}(b)=\Tr_{L/K}(\sg(b))$,  we deduce that $\Tr_{L/K}(\as(\be))=0$. Since 
\[
\Tr_{L/K}(\as(\be))=\Tr_{L/K}(\be^p)-\Tr_{L/K}(\be)=\Tr_{L/K}(\be)^p-\Tr_{L/K}(\be),
\]
this implies that $\Tr_{L/K}(\be)\in \F_p$.

Let us still denote by $\sg\in G$ the restriction of $\sg$ to $M$. The subgroup $\gal(M/L)\sub G$ is cyclic generated by the automorphism $\tau\in G$ determined by
\[
\tau_{\mid L}=1,\qquad \tau(\eta)=\eta+1.
\] 

Since $\as(\sg(\eta)-\eta -\be)=\sg(b)-b-\as(\be)=0$, we have $\sg(\eta)-\eta -\be\in\F_p$. By replacing $\sg$ with $\sg\tau^m$ for a suitable $m\in\N$, we may assume that 
\[
\sg(\eta)=\eta+\be.
\] 
These two automorphisms $\sg,\tau\in G$ commute:
\[
\sg\tau(\al)=\sg(\al)=\al+1=\tau(\al+1)=\tau\sg(\al),
\]
\[
\sg\tau(\eta)=\sg(\eta+1)=\eta+\be+1=\tau(\eta+\be)=\tau\sg(\eta).
\]
Since $G$ has order $p^2$, we deduce that
\begin{equation}\label{power}
	\begin{array}{rcl}
		\sg^p\ne 1&\imp& G=\gen{\sg}\simeq C_{p^2},\\
		\sg^p= 1&\imp& G=\gen{\sg,\tau}\simeq C_p\times C_p.	
	\end{array}	
\end{equation}
The element $\be\in L$ allows us to distinguish between these two possibilities. 

\begin{Cor}\label{structure}
	Suppose that $M/K$ is Galois and let $G=\gal(M/K)$. Then,
	\[
	\begin{array}{ccl}
		G\simeq C_p\times C_p&\ \mbox{if and only if}\ & \Tr_{L/K}(\be)=0.\\
		G\simeq C_{p^2}&\ \mbox{if and only if}\ & \Tr_{L/K}(\be)\ne0.
	\end{array}
	\]
\end{Cor}

\begin{proof}
	Clearly, $\sg^p=1$ if and only if $\sg^p(\eta)=\eta$. Since
	\[
	\sg^p(\eta)=\eta+\be+\sg(\be)+\cdots+\sg^{p-1}(\be)=\eta+\Tr_{L/K}(\be),
	\]
	the statement of the corollary follows from (\ref{power}).
\end{proof}



\section{Subfields of the Hahn field}\label{subsethanh}
Let $k$ be an algebraically closed field of characteristic $p>0$.
For an indeterminate $y$, consider the \textbf{Hahn field} $\Ha$ of formal power series in $y$ with rational exponents and coefficients in $k$. The support of a formal power series is the following set: 
$$s=\sum\nolimits_{q\in\Q}\la_qy^q\ \imp\ \supp\left(s\right)=\{q\in\Q\mid \la_q\ne0\}\sub\Q.$$
The Hahn field $\Ha$ consists of all power series with well-ordered support. The valuation $v=\ord_y$ on $k(y)$ admits a natural extension to $\Ha$, defined by:
\[
v(s):=\min(\supp(s)).
\] 
Clearly,  $v\Ha=\Q$ and $\Ha v=k$.

\subsection{Truncation and brick components of elements in the Hahn field}

For every $s=\sum_{q\in\Q}\la_qy^q\in\Ha$ and $\dta\in\R_\infty$, we define the \textbf{truncation} of $s$ by $\dta$ as
\[
\tr_\dta(s):=\sum\nolimits_{q<\dta}\la_qy^q.
\] 
We agree that $\tr_\dta(s)=0$ if $\dta\le v(s)$. Clearly, for all $\dta\in\Ri$ we have
\[
\tr_\dta(s+t)=\tr_\dta(s)+\tr_\dta(t)\quad\mbox{ for all }\ s,t\in \Ha.
\]

We say that $\dta$ determines a \textbf{strict} truncation of $s$ if $\supp\left(\tr_\dta(s)\right)\subsetneq \supp(s)$.
 For every interval $J=[q,\dta)\sub\R$ we define
\[
\tr_J(s):=\tr_\dta(s)-\tr_q(s),
\]
so that $\supp\left(\tr_J(s)\right)=\supp(s)\cap J$.

For every well-ordered subset $S\sub\R$ we denote by 
$\lm(S)\sub \R_\infty$
the \textbf{set of limits} of $S$. That is, the set of all  $\dta\in\R_\infty$ such that for every open interval $I$ centered at $\dta$ we have $S\cap\left(I\setminus\{\dta\}\right)\ne\emptyset$. 
Clearly, $\lm(S)$ is well-ordered too. 

Note that $\lm(S)=\emptyset$ if and only if $S$ is  finite.
Also, $\infty\in\lm(S)$ if and only if $S$ is unbounded.
Every limit of a family of limits of $S$ is again a limit of $S$: 
 \[
 \lm\left(\lm(S)\right)\sub\lm(S).
 \]
In particular, $\lm(S)$ has always a maximal element. 
For every $s\in\Ha$ we denote
\[\lm(s):=\lm\left(\supp(s)\right)\sub \R_\infty.\]

\begin{Def}\label{defbrick}
We say that $s\in\Ha$ is a \textbf{brick} if   $\lm(s)=\{\dta\}$ for some $\dta\in\R_\infty$ such that   $s=\tr_\dta(s)$. Note that  every strict truncation of $s$ has finite support.
\end{Def}

Equivalently, $s$ is a brick if and only if the order-type of $\supp(s)$ is $\omega$.

Suppose that $s\in\Ha$ has infinite support and let $\dta_{\min}, \dta_{\max}\in \lm(s)$ be the minimal and maximal elements in this set, respectively. Consider the brick:
\[
s_{-\infty}:=\tr_{[v(s),\dta_{\min})}(s),\qquad   \lm(s_{-\infty})=\{\dta_{\min}\}.
\]

For every $\dta\in\lm(s)$, $\dta\ne \dta_{\max}$, let $\dta'$ be its immediate successor in $\lm(s)$. Then, we may consider the brick:   
\[
s_\dta:=\tr_{[\dta,\dta')}(s),\qquad \lm(s_\dta)=\{\dta'\}.
\]
In this way, we obtain a canonical decomposition:
\[
\tr_{\dta_{\max}}(s)=s_{-\infty}+\sum_{\dta\in\lm(s)}s_\dta,
\]
as a sum of bricks with disjoint supports. These summands  are  called the \textbf{brick components} of $s$.

\begin{Def}
Denote $\lm^0(s):=\supp(s)$. For all $n\in\N$, we define \[\lm^{n}(s):=\lm\left(\lm^{n-1}(s)\right).\]
\end{Def}

Note that $\lm^1(s)=\lm(s)$ and $\lm^{n-1}(s)\supsetneq\lm^{n}(s)$, because the minimal element in $\lm^{n-1}(s)$ cannot belong to $\lm^{n}(s)$.

\begin{Lem}\label{limSum}
For all $s,t\in\Ha$ and all $n\in\N$, we have
\[
\left(\lm^n(s)\cup\lm^n(t)\right)\setminus \left(\lm^n(s)\cap\lm^n(t)\right)\sub \lm^n(s+t)\sub \left(\lm^n(s)\cup\lm^n(t)\right).
\]
\end{Lem}

\begin{proof} Let $S=\supp(s)$, $T=\supp(t)$.
Suppose that $\dta\not\in \lm^n(s)\cup\lm^n(t)$. Then, for some  interval $J=[\ep,\dta)$, with $\ep<\dta$, neither $S\cap J$ nor $T\cap J$ contain limits of order $n-1$. Since, \[\tr_J(s+t)=\tr_J(s)+\tr_J(t),\] the set $\supp(s+t)\cap J$ does not contain limits of order $n-1$ either. This implies that  $\dta\not\in \lm^n(s+t)$.
	
	Finally, suppose that $\dta\in\lm^n(t)$, but  $\dta\not\in \lm^n(s)$. Take $\ep<\dta$ such that $S\cap J$ contains no limits of order $n-1$, where $J=[\ep,\dta)$. Then, \[\lm^{n-1}\left(\supp(s+t)\cap J\right)=\lm^{n-1}\left(T\cap J\right).\] Since the latter set admits $\dta$ as a limit, we deduce that $\dta\in \lm^n(s+t)$. 
\end{proof}

\begin{Lem}\label{finiteTr}
 Let $s\in \Ha$ such that $\lm^n(s)=\emptyset$ for some $n\in\N$. Let $\left(\dta_i\right)_{i\in I}\sub\R$ be a strictly decreasing family which is well-ordered with respect to the reversed ordering in $\R$. Then, $\left\{\tr_{\dta_i}(s)\mid i\in I\right\}$ is a finite subset of $\Ha$.  
\end{Lem}

\begin{proof}
We proceed by induction on the minimal $n\in\N_0$ such that $\lm^{n+1}(s)=\emptyset$. If $n=0$, then $\supp(s)$ is finite, and the statement is trivial.
Suppose that $n>0$ is minimal with the property   $\lm^{n+1}(s)=\emptyset$. Then, 
\[
\lm^n(s)=\{\ell_1<\cdots<\ell_m\}
\]
 is a non-empty  finite set. By  taking $\ell_0:=v(s)$ and $\ell_{m+1}:=\infty$, we can decompose:
\[
s=s_1+\cdots +s_{m+1},\qquad s_i=\tr_{[\ell_i-1,\ell_i)}(s),\quad 1\le i\le m+1.
\]
For each $i\le m$, we have $\lm^n(s_i)=\{\ell_i\}$. On the other hand, $s_{m+1}=0$ if $\ell_m=\infty$, or $\lm^n(s_{m+1})=\emptyset$, otherwise.

Clearly, it suffices to prove the lemma for each one of these elements $s_i$. For $s_{m+1}$, the statement follows from the induction hypothesis.  Thus, it suffices to prove the lemma under the assumption that $\lm^n(s)=\{\ell\}$ and $s=\tr_\ell(s)$.

Decompose $I=I_0\sqcup I_\infty$, where
\[
I_0=\{i\in I\mid \dta_i<\ell\},\qquad I_\infty=\{i\in I\mid \dta_i\ge\ell\}.
\]
For all $i\in I_\infty$ we have $\tr_{\dta_i}(s)=s$. Thus, it suffices the prove finiteness for the set
\[
\left\{\tr_{\dta_i}(s)\mid i\in I_0\right\}.
\]
Suppose that $I_0\ne\emptyset$. Since $\left(\dta_i\right)_{i\in I}\sub\R$ is well-ordered with respect to the reversed ordering in $\R$, there exists $\ep=\max(I_0)$. Since $\tr_\ep(s)$ has no limits of order $n$, this element has a finite number of truncations by the family $\left(\dta_i\right)_{i\in I}$, by the induction hypothesis. Since
\[
\tr_{\dta_i}(s)=\tr_{\dta_i}\left(\tr_\ep(s)\right) \quad \mbox{ for all }i\in I_0,
\]
we see that $s$ has a finite number of truncations too.
\end{proof}

\subsection{Subfields of $\Ha$ with no finite limits}

\begin{Def}\label{noLim}
Consider a field $K$ such that  $k(y)\sub K\sub\Ha$. We say that $K$ \textbf{has no finite limits} if $\lm(a)\sub\{\infty\}$ for all $a\in K$. In other words, either  $\supp(a)$ is finite, or $\lm(a)=\{\infty\}$.
\end{Def}

\emph{From now on, we suppose that $K$ is a perfect, Henselian  field,  with no finite limits.  We denote by $\kb$ the algebraic closure of $K$ in $\Ha$. 
}\e

The Henselization of the perfect closure of $k(y)$ and the Newton-Puiseux field 
\[
\bigcup_{n\in \N}k(\!(y^{1/p^n})\!)\,\sub\,  
\bigcup_{n\in \N}k(\!(y^{1/n})\!)  
\]
are concrete examples of fields $K$ satisfying these conditions.

\begin{Lem}\label{immediate}
For all $q\in vK$  we have $y^q\in K$.   
\end{Lem}

\begin{proof}
	Since $q\in vK$, $k$ is algebraically closed and  $K$ is Henselian, the extension $K(y^q)/K$ is immediate and $ [K(y^q)\colon K]$ is a power of $p$. 
	
	Let $m\in\N$ be minimal such that $y^{qm}\in K$.  Since $K$ is perfect, we have $p\nmid m$. Thus, $y^q$ is a root of the Kummer polynomial $x^m-y^{qm}$, which is irreducible in $\kx$ because $k$ contains all roots of $x^m-1$. Hence, $ [K(y^q)\colon K]=m$ and this implies $m=1$.
\end{proof}

\begin{Cor}\label{guess0}
Let  $a\in K^*$. 	For every $q\in\supp(a)$  we have $y^q\in K$.   
\end{Cor}

\begin{proof}
Suppose that there exists some $q\in\supp(a)$ such that $y^q\not\in K$. Since $\supp(a)$ is well-ordered, there exists a minimal $q\in\supp(a)$ with this property.

The truncation $a_0:=\tr_q(a)$ belongs to $K$ because it is a finite sum of terms in $K$. By replacing $a$ with $a-a_0$, we can assume that
\[
a=\la y^q+b,\qquad \la\in k^*, \quad v(b)>q.
\] 
Since $q=v(a)\in vK$, this contradicts Lemma \ref{immediate}
\end{proof}

\begin{Cor}
Let  $a\in K$. For all $\dta\in\R_\infty$ we have $\tr_\dta(a)\in K$.
\end{Cor}
	
\begin{proof}
If $\dta=\infty$, then $\tr_\dta(a)=a\in K$. Otherwise, $\tr_\dta(a)$ is a finite sum of elements in $K$, by Corollary \ref{guess0}.
\end{proof}\e

Recall the function $d_1\colon \kb\setminus K\to \op{Cuts}(\g)$, introduced in Section \ref{secOS}. Take $s\in\kb\setminus K$. If $\supp(s)\sub vK$, then $\lm(s)\ne\emptyset$, because otherwise $s$ would have a finite support and $s$ would belong to $K$ by Lemma \ref{immediate}.
Now, the following computation of $d_1(s)$ follows immediately from the definitions.

\begin{Lem}\label{firstLim}
Suppose that $s\in\kb\setminus K$ satisfies $\supp(s)\sub vK$. Then, $d_1(s)=\dta_{\min}^-$, where $\dta_{\min}:=\min\left(\lm(s)\right)$. 
\end{Lem}


The next observation follows immediately from Lemma \ref{firstLim} and Theorem \ref{OSdepth}.

\begin{Lem}\label{dep1brick}
	Let $s\in\kb\setminus K$ such that $\supp(s)\sub vK$, and let $n=\deg_K s$. Then, $\dep(s)=1$ if and only if the minimal brick component of $s$ has degree $n$ over $K$.
\end{Lem}

\begin{Lem}\label{crucial}
	Let $s\in\kb\setminus K$ such that $\supp(s)\sub vK$ and $\dep(s)=1$. Let $n=\deg_K s$. Then, for every $t\in \kb^*$ there exists some $a\in K$ such that  the minimal  brick component of $t(s-a)$ has  degree $n$ over $K$. 	In particular, if $\deg_K t(s-a)=n$, then $\dep(t(s-a))=1$.
\end{Lem}

\begin{proof}
By Lemma \ref{dep1brick},  $d_1(s)=\dta^-$ for some $\dta\in\Ri$ such that  $\tr_\dta(s)$ is a brick of degree $n$ over $K$. In particular, any strict truncation of $\tr_\dta(s)$ belongs to $K$.

Write $t=\la y^q+t'$, with $\la\in k^*$ and $v(t')>q$. Take
	\[
	a:=\tr_{\dta-\ep}(s)\in K,\quad\mbox{ where }\ \,\ep:=v(t')-q>0.
	\]
Since $d_1(s-a)=\dta^-$ and $v(s-a)\ge \dta-\ep$, we get $t(s-a)=\la y^q(s-a)+t'(s-a)$, with
\[
d_1(\la y^q(s-a))=(q+\dta)^-,\qquad v(t'(s-a))=q+\ep+ v(s-a)\ge q+\dta.
\]
Hence, $\tr_{q+\dta} (t(s-a))=\la y^q \tr_\dta(s-a)=\la y^q \left(\tr_\dta(s)-a\right)$, which is clearly a  brick and has degree $n$ over $K$.
Also,  if $\deg_K t(s-a)=n$, then $\dep(t(s-a))=1$ by Lemma \ref{dep1brick}.
\end{proof}

\subsection{AS extensions of $K$} Any finite subextension $K\sub L\sub\kb$ with $[L\colon K]=p^m$ is necessarily immediate, with $d(L/K)=p^m$. Hence, for our particular choice of $K$, all AS extensions are defect extensions of degree $p$. 
Since the residue field $Kv=k$ is algebraically closed and $K$ is Henselian, the following result follows from \cite{Kuhl}.

\begin{Lem}\label{Split}
If $a\in K$ has $v(a)\ge0$, then $a\in\as(K)$.
\end{Lem}

\begin{Lem}\label{strange}
	Let $a=\sum_{q\in \Q}a_q y^q\in \Ha$ and write it as $\,a=a^-+a_0+a^+$, where
	\[
	a^-=\sum_{q<0}a_q y^q \quad\mbox{and}\quad a^+=\sum_{q>0}a_q y^q.
	\]  
	Then, an element $\al\in \Ha$ such that $\as(\al)=a$ can be constructed as follows:
	\[
	\al=\sum_{i=1}^\infty \left(a^-\right)^{\frac{1}{p^i}}+c_0-\sum_{i=0}^\infty \left(a^+\right)^{p^i},
	\] 
	where $c_0\in k$ is a root of $x^p-x-a_0$.
\end{Lem}

\begin{proof}
 It is obvious that $\as(\al)=a$. On the other hand, it is easy to check that $\supp(\al)\sub\Q$ is a well-ordered subset if $\supp(a)$ is well-ordered.    
\end{proof}


\begin{Cor}\label{d1=0}
Suppose that  $\al\in\Ha$ satisfies $\as(\al)=a\in K\setminus \as(K)$, with $\supp(a)\sub(-\infty,0)$. Then, $\supp(\al)\sub vK$, $d_1(\al)=0^-$, $\al$ is a brick and $\dep(\al)=1$.  
\end{Cor}

\begin{proof}
By Lemma \ref{strange}, $\al\in\left(\sum_{i=1}^\infty a^{\frac{1}{p^i}}\right)+\F_p$ has support contained in $vK$.
 Consider the partial sums
\[
\al_m:=\sum_{i=1}^m a^{\frac{1}{p^i}}\in K,\quad m\in\N.
\]
The values $v\left(\al-\al_m\right)=v(a)/p^{m+1}$ are negative and  arbitrarily close to zero for $m$ large enough. Hence, $0$ is the minimal limit of $\al$. By Lemma \ref{firstLim}, $d_1(\al)=0^-$.  

Obviously, $\al$ is a brick. Finally, Lemma \ref{dep1brick} shows that $\dep(\al)=1$.
\end{proof}

\begin{Teo}\label{mainCompositum}
If $(K,v)$ is Henselian, perfect, has no finite limits and $L/K$ is a compositum of linearly disjoint AS extensions of $K$, then $\dep(L/K,v)=1$.
\end{Teo}

\begin{proof}
For all $i=1,\dots,n$, let $a_i=\as(\al_i)\in K$ and write
\[
a_i=a_i^-+b_i,\qquad \supp(a_i^-)\sub(-\infty,0), \ \supp(b_i)\sub[0,\infty).
\]
Since $v(b_i)\ge0$,  there exists some $e_i\in K$ such that $\as(e_i)=b_i$, by Lemma \ref{Split}. By Lemma \ref{strange}, we can construct an element $\al_i^-\in \Ha$ such that  $\as(\al_i^-)=a_i^-$ as follows
\[
\al_i^-:=\sum_{j=1}^\infty(a_i^-)^{1/p^j}.
\]
By Corollary \ref{d1=0}, we have $d_1(\al_i^-)=0^-$.
Since $\as(\al_i^-+e_i)=a_i=\as(\al_i)$, we have $\al_i^-\in\al_i-e_i+\F_p\sub \al_i+K$, so that $\al_i^-$ generates $K(\al_i)$ as well. Therefore, after eventually replacing each $\al_i$ with $\al_i^-$, we may assume that
\[
d_1(\al_1)=\cdots=d_1(\al_n)=0^-.
\]
Since $\overline {\F_p}\subseteq K$, condition \eqref{equline} is trivially satisfied. Consequently, our result follows from Theorem \ref{generaliza}.
\end{proof}

\section{Conjecture and examples}\label{examples}
In view of Theorem \ref{mainCompositum} we have the following conjecture.
\begin{Con}\label{mainCon}
	Suppose that $(K,v)$ is Henselian, perfect and has no finite limits. An AS tower $M/K$ has  $\dep(M/K,v)=1$ if and only if $M/K$ is a compositum of AS extensions of $K$.
\end{Con}

In this section, we explore this conjecture by analyzing some examples of AS towers of degree $p^2$ and $p^3$.

\subsection{Abhyankar's AS extension.}\label{subsecAbh}
Consider the following element in $\Ha$:
\[
\al=\sum_{n=1}^\infty y^{-1/p^n}.
\]
Clearly, $\as(\al)=y^{-1}$ and $\lm(\al)=\{0\}$. Since $K$ has no finite limits, we see that $L=K(\al)$ is an  AS extension of $K$ admitting $\al$ as an  AS generator. This AS extension was first introduced by 
Abhyankar. 

Our examples of AS towers of degree $p^2$ will be constructed as AS extensions of $L/K$. 
In this section, we review some particular properties of $L/K$.

Note that $\dep(\al)=1$ because  $\al$ is a brick of degree $p$.  
Let us describe the brick components of all elements in $L$. This is obvious if $p=2$ because $L=K\al+K$ in this case.
For $p>2$, it suffices to compute the brick components of $\al^j$, for $2\le j<p$.
Since $\al=y^{-1/p}+\al^{1/p}$, we may write
\begin{equation}\label{binomial}
\al^j=\left(y^{-1/p}+\al^{1/p}\right)^j=
y^{-j/p}+jy^{-(j-1)/p}\al^{1/p}+\cdots+jy^{-1/p}\al^{(j-1)/p}+\al^{j/p}.
\end{equation}
Now, all these summands have disjoint support. Indeed, if we denote
\[
\be_\ell:=\comb{j}{\ell}y^{-(j-\ell)/p}\al^{\ell/p},\quad 0\le\ell\le j,
\]then we have
\[
v\left(\be_\ell\right)=-\dfrac{j-\ell}p-\dfrac{\ell}{p^2},\qquad 
\sup\left(\supp(\be_\ell)\right)=-\dfrac{j-\ell}p.
\]
For $\ell<j$, we have $v(\be_{\ell+1})> -(j-\ell)/p$, because this is equivalent to: $p> \ell+1$.

Now, an iterative application of (\ref{binomial}) leads to
\begin{equation}\label{alj}
\al^j=\sum_{n=0}^\infty \left(\be_0^{1/p^n}+\cdots+\be_{j-1}^{1/p^n}\right),
\end{equation}
and all summands in this expression have disjoint support. 

Suppose that $j=2$ and consider the brick
\[
\be:=\be_0+\be_1=y^{-2/p}+2y^{-1/p}\al^{1/p},
\]
of degree $p$ over $K$. The decomposition (\ref{alj}) determines the following canonical decomposition of $\al^2$ as a sum of brick components of degree $p$, with disjoint support:
\begin{equation}\label{al2}
	\al^2=\be+\be^{1/p}+\cdots+\be^{1/p^n}+\cdots,\qquad \supp(\be^{1/p^n})\sub[-2/p^{n+1},-1/p^{n+1}).
\end{equation}

By a recurrent argumentation, we can use this description to describe the canonical brick components of each summand in (\ref{alj}) for an arbitrary $j<p$.

Let us derive some easy consequences from (\ref{alj}). For every $1<j<p$ we have:
\[
\lm^j(\al^j)=\{0\},\qquad \lm^{j-1}(\al^j)=\bigcup_{n\in\N}\lm^{j-1}(\be_{j-1}^{1/p^n})=\left\{-1/p^n\mid n\in \N\right\}.
\]
Also, note that $\lm^n(\rho)\sub \Q$ for all $\rho\in L$ and all $n\in\N$.

We mention two more consequences, which require some argumentation.

\begin{Lem}\label{trL}
For all $\rho\in L$, $\dta\in\Ri$, we have $\tr_\dta(\rho)\in L$.
\end{Lem} 

\begin{proof}
Take $\rho\in L$. If $\dta=\infty$, then $\tr_\dta(\rho)=\rho\in L$. Suppose that $\dta<\infty$. 	

Take an arbitrary $a\in K^*$. Let us first show that
\begin{equation}\label{claima}
\tr_\dta(\rho)\in L\quad\mbox{for all }\ \dta\in\R\ \sii\ \tr_\dta(a\rho)\in L\quad\mbox{for all }\ \dta\in\R. 
\end{equation}
Indeed, only for the finite number of $q\in \supp(a)$ such that $q<\dta-v(\rho)$, we have $\tr_\dta\left(y^q\rho\right)\ne0$. Since $\tr_\dta$ is a $k$-linear operator, $\tr_\dta(a\rho)$ is a finite sum of elements of the form  
$\tr_\dta(\la y^q\rho)=\la y^q\tr_{\dta-q}(\rho)$, where  $\la\in k^*$. This proves (\ref{claima}).

Hence, since $L=\sum_{0\le j<p}K\al^j$, it suffices to prove the lemma for each power $\al^j$. We can argue by induction on $j$. The statement being obvious for $j=0,1$, let us assume that $j>1$ and the lemma holds for all smaller powers of $\al$. Trivially, if $n\in\N$ and $i<j$, then the lemma holds for $\al^{i/p^n}$ too.
   
If $\dta\ge0$, then $\tr_\dta(\al^j)=\al^j\in L$. Suppose  $\dta<0$. Then, only for a finite number of $n\in\N$ we have 
\[
v\left(\be_0^{1/p^n}\right)=-2/p^{n+1}<\dta.
\] 
Hence, only a finite number of summands in the expression (\ref{alj}) have a non-zero contribution to $\tr_\dta(\al^j)
$. For all these summands, we have $\tr_\dta\left(\be_\ell^{1/p^n}\right)\in L$ by the induction hypothesis.
\end{proof}

\begin{Lem}\label{ord2}
	For all $n\in\N$ denote $J_n:=[-1/p^{n},-1/p^{n+1})$.
	Let $\rho\in L$ such that $\lm^2(\rho)=\{0\}$. Then, there exist $n_0\in \N$ and $\la\in k^*$ such that
\[
	\tr_{J_n}(\rho)\in \la\,\be^{1/p^n}+K\quad \mbox{ for all }\ n\ge n_0.
\]
\end{Lem} 

\begin{proof}
Since $-1/p^n< -2/p^{n+1}$ for all $n\in\N$,
  the decomposition  (\ref{al2}) shows that 
\[
 \tr_{J_n}(\al^2)=\be^{1/p^n}\quad\mbox{ for all }\ n\in\N.
\]

Write $\rho=\sum_{i=0}^{p-1}a_i\al^i$ for some $a_i\in K$. Let $j$ be the largest index such that $a_j\ne0$. Since $\lm^j(\al^j)=\{0\}$, we have $\lm^j\left(a_j\al^j\right)=\supp(a_j)$ and $\lm^j\left(a_i\al^i\right)=\emptyset$ for all $i<j$. By Lemma \ref{limSum}, our assumption $\lm^2(\rho)=\{0\}$ implies that $j=2$ and $a_2=\la$ for some $\la\in k^*$.
Also, since $\lm(\al)=\{0\}$, we have $\lm\left(a_1\al\right)=\supp(a_1)$, so that $a_1\al$ has a finite number of negative limits. Hence, for some $n_0\in\N$ large enough, we shall have $\tr_{[-1/p^{n_0},0)}(a_0+a_1\al)\in K$. 
This ends the proof.
\end{proof}

\subsection{A Galois AS tower of depth two}\label{subsecEx2}
Take $p=2$. It is easy to check that $y^{-1}\not\in\as(K)$ implies $y^{-1}\al\not\in\as(L)$.
Thus, we may consider the AS tower $M=L(\eta)$, where $\eta$ is the AS generator determined by 
\begin{equation}\label{eta2}
 \eta^2+\eta=y^{-1}\al,\qquad \eta=\sum_{n=1}^\infty \left(y^{-1}\al\right)^{1/p^n}.
\end{equation}
Obviously, $\eta$ is not a brick and $\dep(\eta)=2$. More precisely, the brick components of $\eta$ described in (\ref{eta2}) have disjoint support, degree $p$ over $K$,  and an increasing limit: 
\[
\lm\left((y^{-1}\al)^{1/p^n}\right)=\lm\left(y^{-1/p^n}\al^{1/p^n}\right)=\left\{-1/p^n\right\}\quad\mbox{for all } \ n. 
\]
In particular, $
\lm(\eta)=\{-1/p^n\mid n\in\N\}\cup\{0\}.$

Since $\sg(y^{-1}\al)-y^{-1}\al=y^{-1}=\as(\al)$, Lemma \ref{criterion} shows that $M/K$ is a Galois extension. Since $\Tr_{L/K}(\al)=1$, Corollary \ref{structure} shows that $\gal(M/K)$ is a cyclic group of order $4$. According to Conjecture \ref{mainCon}, we should have $\dep(M/K)>1$.

\begin{Lem}\label{Ex2}
$\dep(M/K)=2$.	
\end{Lem} 

\begin{proof}
	Since  $\dep(\eta)=2$, it suffices to show that every $\t\in M\setminus L$ has depth greater than one. 
Let us express the elements in $M$ in terms of the $L$-basis $1,\,\eta$.

Let us first assume that $\t\in \eta+L$. Since $\dep(\t)=\dep(\t+b)$ for all $b\in K$, we can assume that $\t=\eta+a\al$ for some $a\in K$.
Let us show that the first brick component of  $\tr_0(\t)$ has degree $p$ over $K$. Thus, $d_1(\t)<d_p(\t)$ and $\dep(\t)>1$, by Theorem \ref{OSdepth}.

Since we are only interested in $\tr_0(\t)$, we can assume that $a$ has finite support. Since $\lm(\al)=\{0\}$, Lemma \ref{limSum} shows that $\lm(a\al)=\supp(a)$ is finite.
Thus, these limits cannot ``cancel" all negative limits of $\lm(\eta)$. More precisely, there exists $\ep<0$ such that $\tr_{[\ep,0)}(a\al)=0$. Since $\tr_{[\ep,0)}(\eta)$ still has an infinite number of brick components of degree $p$, we have necessarily $d_1(\t)<d_p(\t)$.

Finally, take $\t=\rho\eta+\rho'$, for some  $\rho,\rho'\in L$, $\rho\ne0$. Let us show that the assumption $\dep(\t)=1$ leads to a contradiction. By Lemma \ref{crucial}, this condition would imply the existence of some $a\in K$ (depending on $\rho^{-1}$) such that $\dep\left(\rho^{-1}(\t-a)\right)=1$. Since $\rho^{-1}(\t-a)=\eta+\rho^{-1}(\rho'-a)\in\eta+L$, this contradicts our  previous arguments.
\end{proof}

\subsection{A non-Galois AS tower of depth two }\label{subsecEx3}
Suppose that $p>2$. The following observation is a straightforward consequence of the fact that $K$ has no finite limits.

\begin{Lem}\label{lin}
	The images of $\al$ and $\al^2$ in $L/\as(L)$ are $\F_p$-linearly independent. 	
\end{Lem}

Consider the AS extensions  $M_0=L(\t)$, $M=L(\eta)$, with AS generators 
\begin{equation}\label{eta}
	\t=\sum_{n=1}^\infty \al^{1/p^n},\quad\as(\t)=\al; \qquad \eta=\sum_{n=1}^\infty \al^{2/p^n},\quad \as(\eta)=\al^2.
\end{equation}
By Lemmas \ref{compos} and \ref{lin}, these two AS extensions of $L$ are linearly disjoint. Their compositum $N=L(\t,\eta)=K(\t,\eta)$ is a Galois extension of $L$ with Galois group 
\[
H:=\gal(N/L)\simeq C_p\times C_p,
\]
admitting generators $\tau,\iota\in H$ determined by:
\[
\iota(\t)=\t, \quad \iota(\eta)=\eta+1;\qquad \tau(\t)=\t+1, \quad \tau(\eta)=\eta.
\]

Take any $\sg\in\aut(N/K)$ such that $\sg(\al)=\al+1$. Clearly,
\[\ars{1.2}
\begin{array}{l}
\as(\sg(\t))=\sg(\al)=\al+1=\as(\t+c),\\ 
\as(\sg(\eta))=\sg(\al^2)=\al^2+2\al+1=\as(\eta+2\t+c),
\end{array}
\]
where $c\in k$ satisfies $\as(c)=1$. Hence, there exist $\ell,m\in\F_p$ such that
\[
\sg(\t)=\t+c+\ell,\qquad \sg(\eta)=\eta+2\t+c+m.
\]
By replacing $\sg$ with $\iota^{2\ell-m}\tau^{-\ell}\sg\in\aut(N/K)$, we may assume that $\sg$ satifies  
\begin{equation}\label{sigma}
\sg(\al)=\al+1,\qquad \sg(\t)=\t+c,\qquad \sg(\eta)=\eta+2\t+c.
\end{equation}
Now, consider the following conjugates of $\eta$:
\[
\iota^{m}\sg^\ell(\eta)=\eta+2\ell\t+\ell^2c+m,\quad (\ell,m)\in\F_p^2.
\]
Since we get $p^2$ different conjugates, we see that $N/K$ is the normal closure of $M/K$. In particular, $N/K$ is a Galois extension.

\begin{Lem}\label{Nclosure}
The  Galois group of $N/K$ is a non-commutative semidirect product
	\[
	\gal(N/K)\simeq\left(C_p\times C_p\right)\rtimes  C_p.
	\]
Moreover,	$L/K$ is the unique proper subextension of $M/K$.
\end{Lem}

\begin{proof}
Denote $G=\gal(N/K)$, $H=\gal(N/L)\sub G$ and $C=\gen{\sg}\sub G$. Since $L/K$  is  a Galois extension, we have a  short exact sequence
\[
1\,\lra\,H\,\hooklongrightarrow \,G\,\longtwoheadrightarrow\,\gal(L/K)\,\lra 1,
\]
where $H\hookrightarrow G$ is the inclusion and $G\twoheadrightarrow \gal(L/K)$ is the restriction to $L$. The sequence splits via the group isomorphism $\gal(L/K)\to C$ mapping $\sg_{\mid L}$ to $\sg$.
 
Hence, $G$ is the semidirect product, $G=H\rtimes_\varphi C$, where $\varphi$ is the group homomorphism $\varphi\colon C\,\to\,\aut(H)$ determined by
\[
\varphi(\sg)(x)=\sg x\sg^{-1}\quad\mbox{ for all }\ x\in H.  
\]  
From (\ref{sigma}) one deduces that 
\[
\sg \iota \sg^{-1}=\iota,\qquad \sg\tau\sg^{-1}=\iota^{-2}\tau.
\]
This completely determines the structure of $G$. Note that the center of $G$ is the subgroup generated by $\iota$ and $G=\gen{\sg,\tau}$ is generated by $\sg$ and $\tau$.

The subfield $M\sub N$ is the field fixed by $\tau$, while $L$ is the field fixed by the subgroup  $\gen{\tau,\iota}$.
By Galois theory, $L/K$ is the unique proper subextension of $M/K$ if and only if $\gen{\tau,\iota}$ is the unique proper subgroup of $G$ containing $\tau$. Checking this property  is an easy exercise.
\end{proof}\e

The aim of this section is to show that $\dep(M/K)=2$. To this end, we shall show that every element in $ M=L+L\eta+\cdots+ L\eta^{p-1}$ with positive degree as a polynomial in $\eta$ with coefficients in $L$, has depth greater than one. 

Let us start by computing the brick components of $\eta$. Consider the brick  of degree $p$ that we introduced in Section \ref{subsecAbh}:
\[
\be:=y^{-2/p}+2y^{-1/p}\al^{1/p}\in L.
\]
The decomposition of $\al^2$ given in (\ref{al2}), combined with (\ref{eta}), yield the following decomposition of $\eta$ as a sum  of brick components of degree $p$ with disjoint support:
\begin{equation}\label{brickEta}
\eta=\beta^{1/p}+2\beta^{1/p^2}+\cdots +n\be^{1/p^n}+\cdots	
\end{equation} 
In particular, $\dep(\eta)=2$ (by Theorem \ref{OSdepth}) and $\lm^2(\eta)=\{0\}$.
Also, 
\begin{equation}\label{etaJn}
\tr_{J_n}(\eta)=n\be^{1/p^n},\qquad J_n:=[-1/p^{n},-1/p^{n+1}).
\end{equation}

\begin{Lem}\label{j=1}
All elements in $\eta+L$ have depth greater than one.	
\end{Lem}

\begin{proof}
	Take $\rho\in L$. If $\tr_{-1/p}(\rho)$ has some brick component of degree $p$, then $\dep(\eta+\rho)>1$ because $\tr_{-1/p}(\eta+\rho)=\tr_{-1/p}(\rho)$  has some brick component of degree $p$. Thus, we may assume that $\tr_{-1/p}(\rho)$ belongs to $K$.

	Suppose that all $n\be^{1/p^n}$, brick components of $\eta$ of degree $p$,  are ``cancelled" by the addition of $\rho$. By (\ref{etaJn}), this would imply
\[
\tr_{J_n}(\eta+\rho)\in K\quad \mbox{ for all }\ n\in\N.
\]
By Lemma \ref{ord2}, there would exist $n_0\in\N$ and $\la\in k^*$ such that	
\[
(n+\la)\be^{1/p^n}\in K\quad\mbox{ for all }\ n\ge n_0,
\]
which is impossible, because $\be\in L\setminus K$ and $\la\in k$ is constant.
\end{proof}\e

\begin{Lem}\label{corCrucial}
Take any $j\in\N$, $0<j<p$. If all elements in $\eta^j+L\eta^{j-1}+\cdots+L$ have depth greater than one, then all elements  in $L\eta^j+L\eta^{j-1}+\cdots+L$ have depth greater than one.	
\end{Lem}

\begin{proof}
Take $\zeta=\rho\eta^j+\ga$, with $\rho\in L^*$ and $\deta(\ga)<j$. If $\dep(\zeta)=1$, then  there exists   $a\in K$ (depending on $\rho^{-1}$) such that $\dep\left(\rho^{-1}(\zeta-a)\right)=1$, by Lemma \ref{crucial}. Since $\rho^{-1}(\zeta-a)=\eta^j+\rho^{-1}(\ga-a)\in\eta^j+L\eta^{j-1}+\cdots+L$, this contradicts our assumptions.
\end{proof}\e

For every $\zeta\in M$, let us denote by $\deg_\eta(\zeta)$ the degree of $\zeta$ as a polynomial in $\eta$ with coefficients in $L$, of degree less than $p$. 
So far, we have seen that all elements $\zeta\in M$ with $\deg_\eta(\zeta)=1$ have $\dep(\zeta)>1$. From now on, we fix an index $j\in\N$ such that $1<j<p$. We shall use an inductive argument to show that all elements $\zeta\in M$ with $\deta(\zeta)=j$  have depth greater than one too,  
assuming that this holds for all elements of a smaller degree in $\eta$.

Our first step is to obtain information about the brick components of $\eta^j$ in a similar way as we did in Section \ref{subsecAbh} for $\al^j$.

Focussing  in the decomposition (\ref{brickEta}), for all $ n\in\N$ we define
\begin{equation}\label{etaj}
\eta=\be^{1/p}+2\be^{1/p^2}+\cdots +(n-1)\be^{1/p^{n-1}}+\eta_n,\qquad \eta_n:=\sum_{N=n}^\infty N\be^{1/p^N}.
\end{equation}

Now, we apply Newton's formula:
\[
\eta^j=\left(\be^{1/p}+\eta_2\right)^j=\be^{j/p}+j\be^ {(j-1)/p}\eta_2+\cdots+j\be^{1/p}\eta_2^{j-1}+\eta_2^j.
\] 
An iteration of this formula leads to the following relationship for all $n\in\N$:
\[
\eta_n^j=\ga_n+\eta_{n+1}^j,\qquad \ga_n:=n^j\be^{j/p^n}+jn^{j-1}\be^ {(j-1)/p^n}\eta_{n+1}+\cdots+jn\be^{1/p^n}\eta_{n+1}^{j-1},
\]
where we agree that $\eta_1=\eta$. This leads to a decomposition
\begin{equation}\label{gammas}	
\eta^j= \sum_{n\in\N}\ga_n.
\end{equation}
Now,  these summands $\ga_n$ are neither bricks nor have disjoint supports. Indeed, whenever $n\ne0$ in $k$, we have
\[
\supp(\ga_n)\sub [-2j/p^{n+1},-1/p^{n+1}),\qquad v\left(\eta_{n+1}^j\right))=-2j/p^{n+2}.
\]
Thus,  the $\ga_n$'s have disjoint supports  only for $j\le (p-1)/2$.
We shall rearrange some terms, in order to get a decomposition of $\eta^j$ in which all summands have disjoint support. To this purpose, we need a relevant observation. 

\begin{Lem}\label{CompsTrunc}
Take $j\in\N$ such that $0<j<p$.  For all $\rho\in L$, $\dta\in\R$, we have
\begin{equation}\label{trntrn}
\tr_\dta(\rho\eta^j)\in \tr_\dta(\rho)\eta^j +L\eta^{j-1}+\cdots +L. \end{equation}
\end{Lem}

\begin{proof}
Consider the following particular case of (\ref{trntrn}), corresponding to $\rho=1$:\e

\noindent{\bf Claim. }For all $\ep<0$, we have $\tr_\ep(\eta^j)\in L\eta^{j-1}+\cdots +L$.\e

The Claim is true for $j=1$, because the decomposition (\ref{brickEta}) and Lemma \ref{trL} show that $\tr_\ep(\eta)\in L$. Also, we may consider that (\ref{trntrn}) is an empty statement for $j=0$. 

We proceed by induction. We shall show that, assuming that the Claim holds for $j$ and (\ref{trntrn}) holds for $j-1$, then (\ref{trntrn}) holds for $j$ and the Claim holds for $j+1$.
	
	
	Write $\rho=\sum_{i\in I}\la_i y^{q_i}$ for some $\la_i\in k^*$, $q_i\in\Q$. Let $i_1\in I$ be the minimal index such that $q_{i_1}>\dta$. 
	Since $\max\left(\lm(\eta^j)\right)=0$, we have $\max\left(\lm(\la_i y^{q_i}\eta^j)\right)=q_i\le\dta$ for all $i<i_1$, so that  $\tr_\dta(\la_i y^{q_i}\eta^j)=\la_i y^{q_i}\eta^j$.
	By Lemma \ref{trL},
	\[
	\tr_\dta\left(\sum_{i<i_1}\la_i y^{q_i}\eta^j\right)=\sum_{i<i_1}\la_i y^{q_i}\eta^j=\tr_{\dta}(\rho)\eta^j\in L\eta^j.
	\]
	Thus, we need only to show that 
	\[
	\zeta:=\tr_\dta\left(\sum_{i\ge i_1}\la_i y^{q_i}\eta^j\right)\in L\eta^{j-1}+\cdots+L.
	\]
	For every summand, the Claim for $j$ shows that
	\[
	\tr_\dta\left(\la_i y^{q_i}\eta^j\right)=\la_i y^{q_i} \tr_{\dta-q_i}(\eta^j)\in  L\eta^{j-1}+\cdots+L.
	\] 
	By Lemma \ref{finiteTr}, these truncations determine a finite number of elements in $M$, all of them of degree less than $j$ in $\eta$.
	Let $\zeta_1,\dots,\zeta_m$ be all these truncations, ordered by decreasing end of  support.
	Take
	\[
	i_1<i_2<\cdots <i_m,
	\]
	where each $i_\ell\in I$ is minimal with the property that   $\tr_{\dta-q_{i_\ell}}(\eta^j)=\zeta_\ell$. We can express $\zeta$ as a finite sum:
	\[
	\zeta=\sum_{\ell=1}^{m-1}\tr_{[q_{i_\ell},q_{i_{\ell+1}})}(\rho)\zeta_i.
	\]
	By Lemma \ref{trL}, $\tr_{[q_{i_\ell},q_{i_{\ell+1}})}(\rho)\in L$, so that $\zeta$ belongs to  $L\eta^{j-1}+\cdots+L$.

Finally, the Claim for $j+1$ follows  from the decomposition (\ref{gammas}) for $j+1$. Indeed,  $\deta(\ga_n)=j$ for all $n$, and only a finite number of $\ga_n$ will have $\tr_\ep(\ga_n)\ne0$. By  (\ref{trntrn}) for $j$, all these truncations belong to $M$ and have degree in $\eta$ less than $j+1$.  
\end{proof}\e

Let us proceed to the promised rearrangement in the decomposition (\ref{gammas}). Note that, for all $n\in\N$ we have 
\[
v\left(\eta_{n+2}^j\right)=-2j/p^{n+3}>-1/p^{n+1},
\]
because $p^2>2j$. Hence, for $N>n+1$, the summands $\ga_N$ have support disjoint with $\supp(\ga_n)$. Also, the last summand of $\ga_{n+1}$ does not contribute to $\supp(\ga_n)$ either:
\[
v\left(j(n+1)\be^{1/p^{n+1}}\eta_{n+2}^{j-1}\right)=-\dfrac2{p^{n+2}}-\dfrac{2(j-1)}{p^{n+3}}>-1/p^{n+1},
\] 
because $p^2>2p+2(j-1)$. Thus, for all $n>1$, we consider:
\[
\xi_n:=\tr_{-1/p^{n}}\left(n^j\be^{j/p^n}+jn^{j-1}\be^ {(j-1)/p^n}\eta_{n+1}+\cdots+\comb{j}2n^2\be^{2/p^n}\eta_{n+1}^{j-2}\right).
\]
and we substract this part from each $\ga_n$ and add it to $\ga_{n-1}$:
\[
B_n:=\ga_n-\xi_n+\xi_{n+1} \quad \mbox{ for all }\ n\in\N,
\]
where we agree that $\xi_1=0$. 

Note that $\eta_n\in\eta+L$ for all $n$. Hence, Lemma \ref{CompsTrunc} shows that all $\xi_n\in M$ have degree smaller than $j-1$ in $\eta$.
Thus, these ``blocks" $B_n$ belong to $M$ and have $\deta(B_n)=j-1$, with leading coefficient
$jn\be^{1/p^n}$ (that of $\ga_n$). Moreover, one checks easily that
\[
\supp(B_1)\sub[-2j/p^2,-1/p^2),\qquad \supp(B_n)\sub J_n:=[-1/p^n,-1/p^{n+1}),
\]
for all $n>1$. We have proven  the following observation.

\begin{Lem}\label{CompsEtaj}
For all $1<j<p$, the element $\eta^j$ admits a decomposition as a sum of elements in $M$ of disjoint support
\[
\eta^j= \sum_{n\in\N}B_n,
\]
such that $B_n\in jn\be^{1/p^n}\eta^{j-1}+L\eta^{j-2}+\cdots+L$.
\end{Lem}

We are ready to prove the main result of this section.

\begin{Lem}\label{theend}
	$\dep(M)=2$.
\end{Lem}

\begin{proof}
Lemma \ref{CompsEtaj} shows that $\dep(\eta^j)>1$, because $\dep(B_1)>1$ by the  induction hypothesis.
	
	By Lemma \ref{corCrucial}, in order to prove the lemma we need only to show that all elements in $\eta^j+L\eta^{j-1}+\cdots+L$ have depth greater than one (in the spirit of Lemma \ref{j=1}). 
	
	Take $\zeta=\eta^j+\rho\eta^{j-1}+\xi$ for some $\rho\in L$ and $\xi\in M$ such that $\deta(\xi)<j-1$. 
	
	If $\tr_{-2j/p^2}(\rho)\ne0$, then $\dep(\zeta)>1$. Indeed, since  $\tr_{-2j/p^2}(\eta^j)=0$, Lemma \ref{CompsTrunc} shows that
	\[
	\tr_{-2j/p^2}(\zeta)=\tr_{-2j/p^2}(\rho)\eta^{j-1} + \ \mbox{lower terms},
	\]
	and this element has depth greater than one by the induction hypothesis.
	
	Thus, we may assume that $\tr_{-2j/p^2}(\rho)=0$. Now, suppose that $\rho\eta^{j-1}+\xi$ cancels all blocks $B_n$ of the decomposition of $\eta^j$ given in Lemma \ref{CompsEtaj}. This amounts to 
\[
B_1+\tr_{[-2j/p^2, -1/p^2)}(\rho\eta^{j-1}+\xi)\in K,\qquad B_n+\tr_{J_n}(\rho\eta^{j-1}+\xi)\in K,
\]	
for all $n>1$. Since $j-1>0$, this implies that the leading coefficients of these sums as polynomials in $\eta$ vanish. By Lemmas \ref{CompsTrunc} and \ref{CompsEtaj}, we get:
 \[
 j\be^{1/p}+\tr_{[-2j/p^2,-1/p^2)}(\rho)=0,\quad \mbox{ and}
 \]
\[
 jn\be^{1/p^n}+\tr_{J_n}(\rho)=0,\quad \mbox{ for all }n>1.
 \]
This implies
\[
\tr_0(\rho)=\tr_{[-2j/p^2,-1/p^2)}(\rho)+\sum_{n=2}^\infty	\tr_{J_n}(\rho)=-\sum_{n\in \N} jn\be^{1/p^n}=-j\eta.
\]
This contradicts Lemma \ref{trL}, because $\tr_0(\rho)$ belongs to $L$.
\end{proof}

\subsection{Depth of the tower $N/K$ of degree $p^3$}\label{Exemp3}

In this section, we prove that \[\dep(N/K)=2.\] Since $\gal(N/K)$ is not commutative, $N/K$ cannot be the compositum of three linearly disjoint $\as$-extensions of $K$. Hence, the fact that $\dep(N/K)=2$ supports Conjecture \ref{mainCon} too.  
We keep the notation of Section \ref{subsecEx3}:
\[
L=K(\al),\quad M_0=K(\t),\quad M=K(\eta),\quad N=K(\t,\eta),
\]
where $\as(\al)=y^{-1}$, $\as(\t)=\al$ and $\as(\eta)=\al^2$.

Take $\om:=\t+y\eta\in N$. It is easy to check that $N=K(\om)$, because $\om$ has $p^3$ different conjugates under the action of $\gal(N/K)$. Also, $\dep(\om)=2$ because:  \footnote{Note that $\dep(\eta+y\t)=3$, because this element has $d_1=(-1/p)^-$, $d_p=0^-$, $d_{p^2}=1^-$.}
\[
d_1(\om)=d_p(\om)=0^-, \quad d_{p^2}(\om)=1^-. 
\]

Thus,  the equality  $\dep(N/K)=2$ follows from the following result.

\begin{Lem}\label{dNK}
Every element in $ N=M_0(\eta)$, having a positive degree as a polynomial in $\eta$ with coefficients in $M_0$, has depth greater than one over $K$.  	
\end{Lem}

The proof of Lemma \ref{dNK} can be obtained by mimicking all results of Section \ref{subsecEx3}, but replacing $M=L(\eta)$ with $N=M_0(\eta)$. For instance,  Lemma \ref{corCrucial} is obviously true if we replace $L$ with $M_0$. Therefore, Lemma \ref{dNK} follows from the following result.

\begin{Lem}\label{dNK2}
Take $0<j<p$. Every element in $\eta^j+M_0\eta^{j-1}+\cdots +M_0$, has depth greater than one over $K$.  	
\end{Lem}

This statement can be proved by induction on $j$.
In order to prove Lemma \ref{dNK2} for $j=1$, we need to analyze some properties of the field $M_0=K(\t)$.

\subsubsection{The field $M_0=K(\t)$} Recall the definition of $\t$ given in (\ref{eta}):
\begin{equation}\label{theta}
\t=\sum_{n=1}^\infty \al^{1/{p^n}} = y^{-1/p^2}+2y^{-1/p^3}+\cdots +n y^{-1/p^{n+1}}+\cdots 	
\end{equation}
 
 Clearly, $\lm(\t)=\{0\}$. Now, as we did in Section \ref{subsecAbh} for the powers of $\al$, we need to compute  for every exponent $j$, $1<j<p$, explicit decompositions: 
 \begin{equation}\label{gammast}	
 \t^j=\ga_1+\cdots +\ga_n+\cdots,
 \end{equation} 
 where each $\ga_n\in M_0$ has degree smaller than $j$ as a polynomial  in $\t$ with coefficients in $M_0$. Moreover, all  $\ga_n$
 will have disjoint supports.
 
 For all $n\in \N$, let us write
 \begin{equation}\label{tj}
 	\t=y^{-1/p^2}+2y^{-1/p^3}+\cdots +(n-1) y^{-1/p^{n}}+\t_n ,\qquad \t_n:=\sum_{N=n}^\infty N y^{-1/p^{N+1}}.
 \end{equation}
 Note that $\t_1=\t$ and, more generally, $\t_n\in\t+K$ for all $n\in\N$. 
  Now, for $1<j<p$ we define
 \[
 \ga_n:= \t_n^j-\t_{n+1}^j=\left(ny^{-1/p^{n+1}}+\t_{n+1}\right)^j-\t_{n+1}^j.
 \] 
 This leads to the decomposition (\ref{gammast}).
By  Newton's formula, we get
  \[
\ga_n:=n^jy^{-j/p^{n+1}}+jn^{j-1}y^ {-(j-1)/p^{n+1}}\t_{n+1}+\cdots+jn y^{-1/p^{n+1}}\t_{n+1}^{j-1}.
 \]
Note that $\ga_n=0$ if $p\mid n$. 
 Now,  a recurrent argument shows that  
 \begin{itemize}
 	\item \ $\lm^j(\t^j)=\{0\},\qquad \lm^{j-1}(\t^j)=\{-1/p^{n+1}\mid n\in\N,\ p\nmid n\}$.
 	\item  \ If $p\nmid n$, then $\supp(\ga_n)\sub[-j/p^{n+1},-1/p^{n+1})$.
 \end{itemize}
Since $j<p$, we see that the sets $\supp(\ga_n)$ are all disjoint.

\begin{Lem}\label{trnM0}
For every $\xi\in M_0$ and $\dta\in\Ri$, we have $\tr_\dta(\xi)\in M_0$.	
\end{Lem}

\begin{proof}
	We can prove a more precise statement. Namely, for all $\rho\in L$ and every exponent $0<j<p$, we have 
	\begin{equation}\label{mimict}	
	\tr_\dta(\rho\t^j)\in \tr_\dta(\rho)\t^j+L\t^{j-1}+\cdots+L.
	\end{equation}
This proves the lemma because $\tr_\dta(\rho)$ belongs to $L$ by Lemma \ref{trL}.

In order to prove (\ref{mimict}), we can mimic the proof of Lemma \ref{CompsTrunc}, just by replacing the decompositions (\ref{eta}) and (\ref{etaj}) for the powers of $\eta$, with the analogous decompositions (\ref{theta}) and (\ref{tj}) for the powers of $\t$.	
\end{proof}\e

The proof of the following result is straightforward.

\begin{Lem}\label{Lmalt}
	Let $i,j\in\N$ such that $0\le i,j<p$. Then, 
$\lm^{i+j}(\al^i\t^j)=\{0\}$. 
\end{Lem}

Finally, we need an analogous to Lemma \ref{ord2} for the field $M_0$.

\begin{Lem}\label{Mord2}
	Let $\rho\in M_0$ such that $\lm^2(\rho)=\{0\}$. Then, there exist $n_0\in \N$ and $(\ka,\la,\mu)\in k^3$, $(\ka,\la,\mu)\ne(0,0,0)$, such that
	\begin{equation}\label{Mstatement}
		\tr_{J_n}(\rho)\in\ka\,\be^{1/p^n}+(2n\la+\mu)\,y^{-1/p^{n+1}}\t_n+K\quad \mbox{ for all }\ n\ge n_0,
	\end{equation} 
	where $J_n=[-1/p^{n},-1/p^{n+1})$.
\end{Lem} 

\begin{proof}
Any $\rho\in M_0$ can be written as $\rho=\sum_{0\le i,j<p}a_{i,j}\al^i\t^j$, for some unique $a_{i,j}\in K$. If $\lm^2(\rho)=\{0\}$, then Lemma \ref{Lmalt} shows that 
\begin{itemize}
	\item \ $a_{i,j}=0$ for all pairs $(i,j)$ such that $i+j>2$.
	\item \ $a_{2,0},\,a_{0,2},\,a_{1,1}\in k$.
\end{itemize}
 Also, the element $\rho':=a_{0,0}+a_{1,0}\al+a_{0,1}\t$ has only a finite number of limits in the interval $(-\infty,0)$; hence, $\tr_{J_n}(\rho')$ belongs to $K$, for  all sufficiently large $n$.   

Now, define $\ka:=a_{2,0},\,\la:=a_{0,2},\,\mu:=a_{1,1}\in k$.
From the decompositions (\ref{al2}), and  (\ref{tj}) for $j=2$, we deduce that
\[
\tr_{J_n}(\ka\al^2)=\ka\be^{1/p^n},\]
\[
\tr_{J_n}(\la\t^2)=\la\ga_n=\la\left(n^2y^{-2/p^{n+1}}+2n y^{-1/p^{n+1}}\t_{n+1}\right).
\]
On the other hand, from the decomposition $\al\t=\sum_{n\in\N} y^{-1/p^n}\t$, one deduces that 
\[
\tr_{J_n}(\mu\al\t)=\mu y^{-1/p^{n+1}}\t_n.
\]
Since $\t_{n+1}\in\t_n+K$, this proves the lemma.  
\end{proof}

\subsubsection{Proof of Lemma \ref{dNK2}}
Let us first deal with the case $j=1$.

\begin{Lem}\label{j=1t}
Every element in $\eta+M_0$ has depth greater than one over $K$.	
\end{Lem}

\begin{proof}
	Arguing as in the proof of Lemma \ref{j=1}, it suffices to check that for any $\rho\in M_0$ with $\lm^2(\rho)=\{0\}$, it cannot exist any $n_0\in\N$ such that 
	\[
	\tr_{J_n}(\eta)+\tr_{J_n}(\rho)\in K\quad\mbox{ for all }\ n\ge n_0.
	\]
By Lemma \ref{Mord2}, there would exist $(\ka,\la,\mu)\in k^3$, $(\ka,\la,\mu)\ne(0,0,0)$ such that	
\begin{equation}\label{ecs}
n\be^{1/p^n}+\ka\,\be^{1/p^n}+(2n\la+\mu)\,y^{-1/p^{n+1}}\t_n\in K\quad\mbox{ for all }\ n\ge n_0.
\end{equation}
Since $\be\in L\setminus K$ and $\t_n\in M_0\setminus L$, this condition cannot hold unless 
\[
\ka+n=0 \ \mbox{ and }\ 2n\la+\mu=0 \quad\mbox{ for all }\ n\ge n_0,
\]
which is impossible because $(\ka,\la,\mu)\in k^3$ are constant.
\end{proof}

\begin{Lem}\label{CompsTrunct}
For all $\rho\in M_0$, $\dta\in\Ri$ and every exponent $0<j<p$, we have 
\[	
	\tr_\dta(\rho\eta^j)\in \tr_\dta(\rho)\eta^j+M_0\eta^{j-1}+\cdots+M_0.
\]	
\end{Lem}

\begin{proof}
	This follows by mimicking the arguments in the proof of Lemma \ref{CompsTrunc}. We must only replace everywhere $L$ with $M_0$ and appeal to Lemma \ref{trnM0} instead of Lemma \ref{trL}.
\end{proof}\e

In the same vein, the proof of Lemma \ref{dNK2} follows from mimicking the proof of Lemma \ref{theend}, just by replacing  everywhere $L$ with $M_0$.


\begin{thebibliography}{}



\bibitem{OM}M. Alberich-Carrami$\tilde{\mbox{n}}$ana, J. Gu\`ardia, E. Nart, A. Poteaux, J. Ro\'e M. Weimann, \emph{Polynomial Factorization Over Henselian Fields}, Found. Comput. Math. \textbf{25}  no. 2 (2025), 631-–681.




\bibitem{B}A. Blaszczok, \emph{Distances of elements in valued field extensions}, Manuscripta Math. {\bf 159} (2019), 397--429.


\bibitem{CKR} S. D. Cutkosky, F.-V. Kuhlmann, A. Rzepka, \emph{On the computation of Kähler differentials and characterizations of Galois extensions with independent defect}, Math. Nachr. \textbf{298} (2025), no. 5, 1549--1577.

\bibitem{Nov4} S. D. Cutkosky, H. Mourtada, \emph{Defect and local uniformization},
Rev. R. Acad. Cienc. Exactas Fís. Nat. Ser. A Mat. RACSAM \textbf{113} (2019), no. 4, 4211--4226.








	
















\bibitem{GabRam} O. Gabber, L. Ramero, \emph{Almost ring theory}, Lecture Notes in Mathematics 1800, Springer-Verlag, Berlin, 2003.

\bibitem{Kuhl}F.-V. Kuhlmann, \emph{A classification of Artin-Schreier defect extensions and characterizations of defectless fields}, Ill. J. Math. {\bf 54} (2010), 397--448.

\bibitem{KR} F.-V. Kuhlmann, A. Rzepka, \textit{The valuation theory of deeply ramified fields and its connection with defect extensions}, Trans. Amer. Math. Soc. \textbf{376} (2023), 2693--2738.



\bibitem{Lorenz} F. Lorenz, \emph{Algebra. Vol. II. Fields with structure, algebras and advanced topics}, Universitext. Springer, New York, 2008. 











\bibitem{KP} E. Nart, \emph{Key polynomials over valued fields}, Publ. Mat. {\bf 64} (2020), 195--232.

\bibitem{MLV} E. Nart, \emph{Mac Lane-Vaqui\'e chains of valuations on a polynomial ring}, Pacific J. Math. {\bf 311-1} (2021), 165--195.

\bibitem{OS} E. Nart, \emph{Okutsu sequences in Henselian valued fields}, Polynesian J. Math. {\bf 1} (2024), no 4, 1–27.






\bibitem{NN25} E. Nart, J. Novacoski,  \emph{Depth of extensions of valuations}, arXiv:2503.00850 (2025).

\bibitem{NNP} E. Nart, J. Novacoski, G. Peruginelli, \emph{A topological approach to key polynomials}, J. Algebra \textbf{684} (2025), 280--307.


\bibitem{NS2026} J. Novacoski, M. Spivakovsky, \emph{K\"ahler differentials, pure extensions and minimal key polynomials}, arXiv:2311.14322 (2023).













\bibitem{Vaq}
M. Vaqui\'e, \emph{Extension d'une valuation}, Trans. Amer. Math. Soc.  {\bf 359} (2007), no. 7, 3439--3481.







\end{thebibliography}
\end{document}